\documentclass{amsart}
\usepackage{amssymb,graphicx}
\newtheorem{theorem}{Theorem}[section]
\newtheorem{conjecture}[theorem]{Conjecture}
\newtheorem{corollary}[theorem]{Corollary}

\newtheorem{lemma}[theorem]{Lemma}

\newtheorem{proposition}[theorem]{Proposition}

\theoremstyle{remark}

\newcommand{\cgE}{\mathcal{E}}

\newcommand{\cgH}{\mathcal{H}}
\newcommand{\cgR}{\mathcal{R}}
\newcommand{\KriNes}{K\v{r}\'{i}\v{z} and Ne\v{s}et\v{r}il}
\newcommand{\NesRod}{Ne\v{s}et\v{r}il and R\"{o}dl}
\newcommand{\eye}{\operatorname{eye}}

\newcommand{\girth}{\operatorname{girth}}

\begin{document}

\title[POSETS AND GRAPHS]{Incidence Posets and Cover Graphs}

\author[W.~T.~TROTTER]{William T. Trotter}

\address{School of Mathematics\\
  Georgia Institute of Technology\\
  Atlanta, Georgia 30332\\
  U.S.A.}

\email{trotter@math.gatech.edu}

\author[R.~WANG]{Ruidong Wang}

\address{School of Mathematics\\
  Georgia Institute of Technology\\
  Atlanta, Georgia 30332\\
  U.S.A.}

\email{rwang49@math.gatech.edu}

\date{August 5, 2013}



\begin{abstract}
We prove two theorems concerning incidence posets
of graphs, cover graphs of posets and a related
graph parameter.  First, answering
a question of Haxell, we show that the chromatic number of a 
graph is not bounded in terms of the dimension of its 
incidence poset, provided the dimension is at least four.
Second, answering a question of \KriNes, we show that
there are graphs with large girth and large
chromatic number among the class of graphs having
eye parameter at most two.  
\end{abstract}

\maketitle

\section{Introduction}

The \textit{chromatic number} of a graph $G=(V,E)$, denoted $\chi(G)$, is
the least positive integer $r$ for which there is a partition
$V=V_1\cup V_2\cup\dots\cup V_r$ of the vertex set $V$ of $G$
so that $V_i$ is an independent set in $G$, for each $i=1,2,\dots,r$.
A family $\cgR=\{L_1,L_2,\dots,L_t\}$ of linear extensions of a poset
$P$ is a \textit{realizer} of $P$ if $P=\cap\cgR$, i.e.,
$x<y$ in $P$ if and only if $x<y$ in $L_i$ for each $i=1,2,\dots,t$.
The \textit{dimension} of a poset $P$, denoted $\dim(P)$, is
the minimum size of a realizer of $P$.

When $G=(V,E)$ is a graph, the \textit{incidence poset} of $G$,
denoted $P_G$, has $V\cup E$ as its ground set; vertices in $V$ are minimal
elements of $P_G$; edges in $E$ are maximal elements of $P_G$; and a vertex
$x$ is less than an edge $e$ in $P_G$ when $x$ is one of the two endpoints
of $e$.  Alternatively, considering the edges of a graph as two element
subsets of the vertex set, the incidence poset $P_G$ of a graph $G=(V,E)$ 
is just the set $V\cup E$ partially ordered by inclusion.

When $x$ and $y$ are points in a poset $P$ with $x>y$ in $P$, we
say $x$ \textit{covers} $y$ in $P$ when there is no point $z$ with
$x>z>y$ in $P$. Also, we let $G_P$ denote the \textit{cover} graph of $P$.
The graph $G_P$ has the elements of $P$ as vertices with $\{x,y\}$ an
edge in $G_P$ when one of $x$ and $y$ covers the other in $P$.
The \textit{diagram} of $P$ (also called a \textit{Hasse} diagram or 
\textit{order} diagram) is a drawing of the cover graph of $G$ in 
the plane with $x$ higher in the plane
than $y$ whenever $x$ covers $y$ in $P$.

\subsection{Mathematical Preliminaries}

We will find it convenient to work with the following alternative
definition of dimension of an incidence poset, as proposed by
Barrera-Cruz and Haxell~\cite{b:BarHax}, 

\begin{proposition}\label{p:alternate}
Let $G=(V,E)$ be a graph and let $P_G$ be its incidence poset.  Then 
$\dim(P_G)$ is the least positive integer $t$ for which there is
a family $\{L_1,L_2,\dots,L_t\}$ of linear orders on $V$ so that the
following two conditions are satisfied:
\begin{enumerate}
\item If $x$, $y$ and $z$ are distinct vertices of $G$ and
$\{y,z\}$ is an edge in $G$, then there is some $i$ with $1\le i\le t$
for which both $x>y$ and $x>z$ in $L_i$.
\item If $x$ and $y$ are distinct vertices of $G$, then there is
some $i$ with $1\le i\le t$ so that $x>y$ in $L_i$.
\end{enumerate}
\end{proposition}

We note that some authors (see Ho\c{s}ten and Morris~\cite{b:HosMor},
for example) have worked with a notion of \textit{dimension} of a graph, which
uses only the first of the two conditions listed in Proposition~\ref{p:alternate}.
This parameter is at most the dimension of the associated incidence poset.
However, it can be at most one less.  Furthermore, for connected graphs 
with no vertices of degree one, the two parameters agree.

In~\cite{b:KriNes}, \KriNes\ defined a new parameter for
graphs, which we call the \textit{eye parameter}.
Formally, the eye parameter of a graph $G$, denoted $\eye(G)$, 
is the least positive integer $s$ for which there exists a family 
$\{L_1,L_2,\dots,L_s\}$ of linear orders on the
vertex set of $G$ for which if $x$, $y$ and $z$ are three distinct vertices
of $G$ with $\{y,z\}$ an edge of $G$, then there is some $i$ with
$1\le i\le s$ for which $x$ is not between $y$ and $z$ in $L_i$.  In this
definition, it is allowed that $x$ be \textit{above} both $y$ and $z$ or
\textit{below} both $y$ and $z$.  For example, when $G$ is a path, $\eye(G)=1$.
The following elementary proposition is stated for emphasis.

\begin{proposition}\label{p:eye-dim}
Let $G$ be a graph and let $P_G$ be the incidence poset
of $G$. Then $\eye(G)\le \dim(P_G)\le 2\eye(G)$.  
\end{proposition}

\section{Dimension and Chromatic Number}

In this section, we state our two main theorems, including just enough
background discussion to place them in context. Proofs are given in
the two sections immediately following.

To understand the fundamental importance of incidence posets and
dimension, we need only mention the following now classic theorem of 
Schnyder~\cite{b:Schn} (recently, a quite clever and very short proof 
has been provided by Barrera-Cruz and Haxell~\cite{b:BarHax}).

\begin{theorem}\label{t:Schnyder}
Let $G$ be a graph and let $P_G$ be its incidence poset.
Then $G$ is planar if and only if $\dim(P_G)\le 3$.
\end{theorem}

In view of Schnyder's result, it is natural to ask the following two questions:

\medskip
\noindent
\textbf{Question 1.}\quad 
Is the dimension of the incidence poset of a graph bounded
in terms of the chromatic number of the graph?

\medskip
\noindent
\textbf{Question 2.}\quad 
Is the chromatic number of a graph bounded in terms of the
dimension of its incidence poset?

\medskip
The first question was answered in the affirmative by
Agnarsson, Felsner and Trotter in~\cite{b:AgFeTr}, where the 
following asymptotic formula is proved.

\begin{theorem}\label{t:AgFeTr}
If $G$ is a graph, $P_G$ is the incidence poset of $G$ and
$\chi(G)=r$, then $\dim(P_G)=O(\lg\lg r)$.
\end{theorem}

The inequality in Theorem~\ref{t:AgFeTr} is best possible, up to
the value of the multiplicative constant,
since as noted in~\cite{b:AgFeTr} (and by others), the dimension of the
incidence poset $P_r$ of the complete graph $K_r$ is at least $\lg\lg r$,
a statement which follows easily from repeated application of the 
Erd\H{o}s/Szekeres theorem on monotonic sequences.  On the other hand,
Ho\c{s}ten and Morris~\cite{b:HosMor} showed
that it is possible to determine the exact value of the dimension of $P_r$
for surprisingly large values of $r$.  Furthermore, a relatively tight
asymptotic formula is known (see~\cite{b:AgFeTr}):
\[
\dim(P_r)=\lg\lg r+(1/2 +o(1))\lg\lg\lg r.
\]
It follows easily that if $P$ is the incidence poset of a graph
$G$ with $\chi(G)=r$, then
\[
\dim(P)\le 2 \lg\lg r+(1 +o(1))\lg\lg\lg r.
\]

In view of Schnyder's theorem, we know the answer to the
second question is yes, provided the dimension of the incidence
poset is at most three. But, in this paper, we show that in general the
answer to Question~2 is no, by proving the following theorem.

\begin{theorem}\label{t:main-1}
For every $r\ge1$, there exists a graph $G$ with $\chi(G)\ge r$ and
$\dim(P_G)\le 4$.
\end{theorem}

Of course, the inequality $\dim(P_G)\le4$ in Theorem~\ref{t:main-1} will become
tight once $r\ge5$.  

\subsection{Cover Graphs}

The cover graph of a poset is
a triangle-free graph, and some thirty years ago, Rival asked whether there
are cover graphs with large chromatic number.  Bollob\'{a}s noted in~\cite{b:Boll}
that B.~Descartes' classic proof~\cite{b:Desc} of the existence of 
triangle-free graphs with large chromatic number provided a positive answer.  
In fact, this construction shows that for each $r\ge1$, there is a poset 
$P$ of height $r$ so that the chromatic number of the cover graph of $P$ is 
$r$.  In view of the dual form of Dilworth's theorem, this is the minimum value of 
height for which such a poset can possibly exist.  

So with the fundamental question answered, several authors went on to prove
somewhat stronger results.  First, Bollob\'{a}s~\cite{b:Boll} showed that there
are lattices whose cover graphs have arbitrarily large chromatic number.
Second, the construction of \NesRod~\cite{b:NesRod}
for graphs (and hypergraphs) with large girth and large chromatic number also 
implies that for each pair $(g,r)$ of positive integers,
there is a poset  $P$  of height $r$ whose cover
graph $G_P$ has girth at least $g$ and chromatic number~$r$.

In another direction, as is well known, the shift graph $S_n$ consisting of all 
$2$-element subsets of $\{1,2,\dots,n\}$
with $\{i,j\}$ adjacent to $\{j,k\}$ when $1\le i <j < k\le n$ is a
cover graph of an interval order and has chromatic number $\lceil \lg n\rceil$.
In this case, the height of the associated interval order is $n-1$, and it
was shown by Felsner and Trotter in~\cite{b:FelTro} that the height of 
an interval order must be exponentially large in
terms of the chromatic number of its cover graph.
In fact, they conjecture that if $P$ is 
an interval order whose height is at most $2^{r-1}+\lfloor(r-1)/2\rfloor$, 
then the chromatic number of the cover graph of $P$ is at most~$r$.  If true, 
this statement is best possible.  Efforts to 
resolve this conjecture have led to interesting problems on 
hamiltonian paths in the subset lattice (see~\cite{b:BirHow}
and~\cite{b:StrTro} for quite recent
work on this theme).

Finally, we mention the work of \KriNes~\cite{b:KriNes} answering
a question posed by Ne\v{s}et\v{r}il and Trotter by proving the following 
theorem, as this work is central to the results presented here.

\begin{theorem}\label{t:KriNes}
For every $r\ge1$, there exists a poset $P$ with $\dim(P)\le 2$ so that
the chromatic number of the cover graph of $P$ is $r$.
\end{theorem}

If $G$ is the cover graph of poset $P$ with $\dim(P)\le 2$, then $\eye(G)\le 2$, so
as \KriNes\ noted, we have the following immediate corollary.

\begin{corollary}\label{c:KriNes}
For every $r\ge1$, there is a graph $G$
with $\eye(G)\le2$ and $\chi(G) = r$. 
\end{corollary}

The graphs constructed by \KriNes\ in the proof of Theorem~\ref{t:KriNes} 
and Corollary~\ref{c:KriNes}
have girth four.  However, they were able to prove the following extension.

\begin{theorem}\label{t:KriNes-2}
For every pair $(g,r)$ of positive integers, there is a graph $G$
with $\eye(G)\le3$, $\girth(G)\ge g$ and $\chi(G)= r$. 
\end{theorem}

They asked whether this result remains true if we require $\eye(G)\le 2$.
Our second main result will be to answer this question in the affirmative
by proving the following theorem, which is in fact a slightly stronger
result.

\begin{theorem}\label{t:main-2}
For every pair $(g,r)$ of positive integers, there is a poset $P=P(g,r)$ with
cover graph $G=G(g,r)$ so that the height of $P$ is $r$, while
$\girth(G)\ge g$ and $\chi(G) =r$.  Furthermore, there
are two linear extensions $L_1$ and $L_2$ of $P$ witnessing that
$\eye(G)\le 2$.
\end{theorem}

The reader should note that we do not claim that the poset $P(g,r)$
in Theorem~\ref{t:main-2} is $2$-dimensional.  In fact, the dimension
of $P(g,r)$ grows rapidly with $r$, even with $g$ fixed.  We will
return to this issue in the last section of this paper.

\section{Proof of the First Main Theorem}

We first explain why Theorem~\ref{t:main-1} follows as a relatively
straightforward corollary to
Theorem~\ref{t:KriNes}, starting with a lemma which we believe is
of independent interest\footnote{We thank an anonymous referee for
pointing out that our original manuscript included this lemma
implicitly.}.  

\begin{lemma}\label{l:referee}
Let $P$ be a poset, let $G$ be the cover graph of $P$ and
let $Q$ be the incidence poset of $P$. Then
$\dim(Q)\le 2\dim(P)$.
\end{lemma}

\begin{proof}
Let $t=\dim(P)$ and let $\cgR=\{L_i:1\le i\le t\}$ be a
realizer of $P$.  Then for each $i=1,2,\dots,t$, let
$L_i^d$ be the dual of $L_i$, i.e., $x>y$ in $L_i^d$ if
and only if $x<y$ in $L_i$.  We claim that the
family $\cgR^*=\cgR\cup\{L_i^d:1\le i\le t\}$ witnesses
that $\dim(Q)\le 2t$.

To see this, note that the second condition of Proposition~\ref{p:alternate} 
holds since $L_1$ and $L_1^d$ are in the family.
Now let $x$, $y$ and $z$ be distinct vertices with $\{y,z\}$ an
edge in $G$.  Without loss of generality, we take $y<z$ in $P$.
If $x\not<z$ in $P$, then there is some $i$ with
$1\le i\le t$ so that $x>z$ in $L_i$.  This implies
$x>z>y$ in $L_i$.  
So we may assume that $x<z$ in $P$.  Since $\{y,z\}$ is an edge of
the cover graph, we cannot have $y<x$ in $P$.  It follows that
there is some $j$ with $1\le j\le t$ so that $x<y$ in $L_j$.
This implies that $x<y<z$ in $L_j$ and $x>y>z$ in $L_j^d$.
This completes the proof of the lemma.
\end{proof}

We now show how Theorem~\ref{t:main-1} follows as an easy corollary
to Theorem~\ref{t:KriNes}.  Let $r\ge1$ and let $P$ be the poset from
Theorem~\ref{t:KriNes}.  Then let $G=G_P$ be the cover graph of 
the poset $P$, noting that $\chi(G) \ge r$.  
Since $\dim(P)\le 2$, from Lemma~\ref{l:referee}, we know that
the dimension of the incidence poset of $G$ is at most four.

\section{Proof of the Second Main Theorem}

We fix an integer $g\ge 4$ and then argue by induction on $r$.  The basic
idea behind the proof will be to make a minor adjustment to
the construction used by \NesRod\ in~\cite{b:NesRod}.  The cases
$r=1$ and $r=2$ are trivial.
To handle the case $r=3$, we let $n$ be an odd integer with
$n\ge g$.  Then we take $G = G(g,3)$ as an odd cycle with vertex
set $\{a_1,a_2,\dots,a_n\}$, with $\{a_i,a_{i+1}\}$ an edge for each
$i=1,2,\dots,n-1$.  Also, $\{a_n,a_1\}$ is an edge of $G$.
Then we take $P=P(g,3)$ as a poset whose cover graph is $G$ by setting the
following covering relations in $P$:
\[
a_1<a_2<a_3>a_4<a_5>a_6<a_7>a_8<a_9>\dots>a_{n-1}<a_n>a_1.
\]
We then take
\[
L_1=a_1<a_2<a_4<a_3<a_6<a_5<a_8<a_7<\dots<a_{n-1}<a_{n-2}<a_n
\]
and
\[
L_2=a_{n-1}<a_1<a_n<a_{n-3}<a_{n-2}<\dots<a_6<a_7<a_4<a_5<a_2<a_3.
\]
It is easy to see that $L_1$ and $L_2$ are linear extensions of $P$.
Furthermore, the two endpoints of an edge in $G$ occur consecutively
in either $L_1$ or $L_2$, except for the edge $\{a_{n-1},a_n\}$.  
However, only $a_{n-2}$ is between $a_{n-1}$ and $a_n$ in $L_1$.
Also, only $a_1$ is between $a_{n-1}$ and $a_n$ in $L_2$. It follows
that $L_1$ and $L_2$ witness that $\eye(G)\le 2$.

Now suppose that for some $r\ge3$, we have constructed a poset $P=
P(g,r)$ with cover graph $G=G(g,r)$ so that the height of $P$ is $r$, while 
$\girth(G)\ge g$ and $\chi(G) = r$.  Suppose further that $L_1$ and $L_2$ 
are linear extensions of $P$ witnessing that 
$\eye(G)\le 2$.  

We now explain how to construct a poset $Q=P(g,r+1)$ with cover graph $H=
G(g,r+1)$ so that
the height of $Q$ is $r+1$, while $\girth(H)\ge g$ and
$\chi(H)=r+1$.  We will also construct linear extensions
$M_1$ and $M_2$ of $Q$ witnessing that $\eye(H)\le 2$.  As the reader
will sense, there is considerable flexibility in how these steps are
taken, and our approach is an effort to make
the exposition as clear as possible.

Let $A$ denote the vertex set of $G$ and let $n=|A|$. 
Using the results of \NesRod\ as developed in~\cite{b:NesRod},
we know there exists a hypergraph $\cgH$ satisfying the following
conditions: $\cgH$ is a simple $n$-uniform hypergraph; the girth
of $\cgH$ is at least $g$; and the chromatic number of $\cgH$ is $r+1$.
Let $B$ and $\cgE$ denote, respectively, the vertex set and the
edge set of $\cgH$.  In the discussion to follow, we consider each
edge $E\in\cgE$ as an $n$-element subset of $B$.

The poset $Q$ is assembled as follows.  Set $Z=\cgE\times A$.
The ground set of $Q$ will be $B\cup Z$ with all elements of
$B$ maximal in $Q$.  For each edge $E$ in $\cgE$, the elements
of $\{E\}\times A$ determine a subposet of $Q$ which we
will denote $P(E)$.  When $a$ and $a'$ are distinct elements of $A$,
we will set $(E,a)<(E,a')$ in $Q$ if and only if $a<a'$ in $P$.
Accordingly, for each $E\in\cgE$, the subposet $P(E)$ is isomorphic
to $P$.  Also, when $E,E'\in \cgE$ and $E\neq E'$, we make
all elements of $P(E)$ incomparable with all elements of $P(E')$.

We pause to point out that regardless of how the comparabilities
between $B$ and $Z$ are defined in $Q$, for each edge $E\in\cgE$,
the covering edges of $P(E)$ are covering edges in $Q$ and these
edges form a copy of $G$.

We now describe these comparabilities between $B$ and $Z$. This will be done by
prescribing when an element $b\in B$ covers an element $(E,a)\in Z$.
We begin by choosing an arbitrary linear order $L(B)$ on $B$.  
Also, let $\{a_1,a_2,\dots,a_n\}$
be a labelling of $A$ so that $L_1$ is the subscript order, i.e.,
$a_i<a_j$ in $L_1$ if and only if $i<j$.   Next, we fix an edge
$E\in\cgE$ and describe the cover relations between $B$
and $P(E)$.  This process will be repeated for each edge $E\in\cgE$ and
when this step has been completed, the poset $Q$ is fully determined. 
First, when $b\in B-E$,  we make $b$ incomparable to all elements of 
$P(E)$ in $Q$.  Second, let $\{b_1,b_2,\dots,b_n\}$ be the labelling of
the elements of $E$ so that $b_i<b_j$ in $L(B)$ if and only if $i<j$.
Then for each $i=1,2,\dots,n$, we make $b_i$ cover
$(E,a_i)$ in $Q$.  It follows that if $(E,a)\in Z$, then there is
a unique element $b\in B$ so that $b$ covers $(E,a)$ in $Q$.

Now that $Q=P(G,r+1)$ has been defined, we take $H=G(g,r+1)$ as the
cover graph of $Q$, and we pause to show that the height of $Q$ is
$r+1$, while $\girth(H)\ge g$ and $\chi(H)=r+1$.
First, we note that the height of $Q$ is at most $r+1$, since we have 
added $B$ as a set of maximal elements to a family of pairwise disjoint
and incomparable copies of $P$.  On the other hand, once we
have shown that $\chi(H)= r+1$, we will have also shown that
the height of $H$ must be $r+1$, using the dual form of Dilworth's
theorem.  

Second, we note that $\chi(H)\ge r$, since $H$ contains copies of $G$.
On the other hand, it is trivial that we may color all elements of
$Z$ with $r$ colors and use one new
color on the independent set $B$, so that $\chi(H)\le r+1$.
Now suppose that $\chi(H)=r$, and let
$\phi$ be a proper coloring of $H$ using
$r$ colors.  Then since the chromatic number of $\cgH$ is 
$r+1$, there is some edge $E$ of $\cgH$ on which $\phi$ is constant.
This implies that $\phi$ colors the cover graph of $P(E)$ with only
$r-1$ colors, which is impossible. The contradiction shows that
$\chi(H)=r+1$, as desired.

Third, we show that the girth of $H$ is at least $g$.  Consider a cycle
$C$ in $H$.  If there is an edge $E$ of $\cgH$ so that
$C$ is contained entirely within the cover graph of $P(E)$, 
then it has size at least $g$.  So we may assume that $C$ involves
vertices from copies of $P$ associated with two or more edges in $\cgE$.
Now the fact that the covering edges between $Y$ and each $P(E)$ are
formed using a bijection means that once the cycle enters some $P(E)$, it
must pass through at least two vertices before leaving.  So the
girth requirement is satisfied (generously) be the fact that the
girth of $\cgH$ is at least $g$.  In this detail, we point out that we
are using essentially the same idea as in~\cite{b:NesRod}.

Now we turn our attention to the eye parameter.
To complete the proof, we must construct two linear extensions $M_1$ and
$M_2$ of $Q$ witnessing that $\eye(H)\le 2$.  As we remarked previously,
there is considerable flexibility in how this is done.

For each $b\in Y$, let $N(b)$ denote the set of all
elements $(E,a)$ from $Z$ such that $b$ covers $(E,a)$ in $Q$, i.e.,
$N(b)$ is just the neighborhood of $b$ in the cover graph $H$.  Note
that $N(b)$ is an antichain in the poset $Q$.

Let $L(\cgE)$ be an arbitrary linear order on $\cgE$.
We define linear extensions $M_1$ and $M_2$ by the following
rules (starting with the rules for $M_2$):

\begin{enumerate}
\item The restriction of $M_2$ to $B$ is an arbitrary linear order.
In $M_2$ all elements of $Z$ are below all elements of $Y$. Furthermore,
if $(E,a)$ and $(E',a')$ are distinct elements of $Z$, then
$(E,a)<(E',a')$ in $M_2$ if and only if either $E<E'$ in $L(\cgE)$ or 
$E=E'$ and $a<a'$ in $L_2$.  
\item The restriction of $M_1$ to $B$ is the linear order $L(B)$.
In $M_1$, for each $b\in B$, all elements of $N(b)$ will be placed 
in the gap immediately under $b$ and above all other elements (if any) 
of $B$ which are under $b$ in $L(B)$.  The restriction of $M_1$ to 
$N(b)$ will be the dual of the restriction of $M_2$ to $N(b)$. 
\end{enumerate}

We pause to show that $M_1$ and $M_2$ are linear extensions of $Q$, and
we remark that it is enough to show that they both respect the covering
relations in $Q$.  First, we note that for each $i=1,2$, and for
each $E\in\cgE$, if $a$ and $a'$ are distinct elements of $A$, then
$(E,a)<(E,a')$ in $M_i$ if and only if $a<a'$ in
$L_i$.  On the other hand, if $b\in B$, $(E,a)\in Z$ and
$b$ covers $(E,a)$ in $Q$, then $(E,a)\in N(b)$ so it is placed below
$b$ in $M_1$.  Finally, we note that all elements of $Z$ are below
all elements of $B$ in $M_2$.  We conclude that $M_1$ and $M_2$ are linear
extensions of $Q$, as desired.

Finally, we explain why $M_1$ and $M_2$ witness that $\eye(H)\le 2$.
Consider how an edge might possibly trap a vertex in both $M_1$ and
$M_2$.  If the edge is an edge in the cover graph of some $P(E)$, then 
the linear extension $M_2$ forces the vertex to also belong to $P(E)$.  
But the restriction of $M_1$ and $M_2$ to $P(E)$ are just like 
$L_1$ and $L_2$ for $G$, so this situation cannot
lead to a problem.

Similarly, if the edge joins some $b\in Y$ to a vertex $(E,a)$ in
$N(b)$, then the only potential problem is a vertex $(E',a')
\in N(b)$ with $(E,a)<(E',a')<y$ in $M_1$.  However, the rules for
$M_1$ and $M_2$ imply that
$(E',a')<(E,a)<b$ in $M_2$.  This completes the proof of Theorem~\ref{t:main-2}.

\section{Conjectures and Questions}

We have made some effort, without success, to construct a 
poset $P$ with cover graph $G$ so that
the dimension of $P$ is small; the girth of $G$ is
large; and the chromatic number of $G$ is large.
Accordingly, we believe it reasonable to make the following
conjecture.

\begin{conjecture}  For every pair $(g,d)$ of integers, with
$g\ge 5$ and $d\ge1$, there is an integer $r=r(g,d)$ so 
that if $G$ is the cover graph of a poset $P$, $\dim(P)\le d$ 
and $\girth(G)\ge g$, then $\chi(G)\le r$.
\end{conjecture}

In another direction, we return to Proposition~\ref{p:eye-dim}
and make the following conjectures.

\begin{conjecture}
For every $t\ge1$, there is a graph $G$ so that if
$P_G$ is the incidence poset of $G$, then
$\eye(G)=\dim(P_G)$.
\end{conjecture}

\begin{conjecture}
For every $t\ge1$, there is a graph $H$ so that if
$P_H$ is the incidence poset of $H$, then
$\dim(P_H)=2\eye(H)$.
\end{conjecture}

A similar analysis of Lemma~\ref{l:referee}
leads to the following conjectures.

\begin{conjecture}\label{c:referee}
For every $t\ge1$, there exists a poset $P$
so that if $G$ is the cover graph of $P$ and
$Q$ is the incidence poset of $G$, then
$\dim(Q)= 2\dim(P)$.
\end{conjecture}

Clearly, Conjecture~\ref{c:referee} holds when $t\le2$, but we
have not been able to settle the issue for larger values of $t$.
However, our preliminary thoughts on this
conjecture suggest a more extensive line of research.  
For a poset $P$, we call a family $\cgR=\{L_1,L_2,\dots,L_t\}$
of linear extensions of $P$ an \textit{upper-cover
realizer of $P$} provided that whenever $(z,x,y)$ is
an ordered triple of distinct points in $P$ with $z$ covering
both $x$ and $y$, there is some $i$ with
$x>y$ in $L_i$.  The \textit{upper cover dimension} of $P$,
denoted $\dim_{uc}(P)$, would then be the minimum size
of an upper-cover realizer of $P$.

Lower-cover realizers and the lower cover dimension of $P$, denoted
$\dim_{lc}(P)$, would then be defined dually. 
Clearly, $\dim_{uc}(P)\le \dim(P)$ and
$\dim_{lc}(P)\le \dim(P)$.

An attractive feature of these new parameters is that they are monotonic
on subdiagrams of the order diagram of $P$, i.e., if we consider the
diagram $D$ of $P$ as an acyclic orientaton of the cover graph $G$, 
and $D'$ is a subdiagram of $D$, then $D'$ determines a poset $P'$ which is 
a suborder of $P$.  On the one
hand, $P'$ is not necessarily a subposet of $P$ and it is clear that
$\dim(P')$ may be the same as $\dim(P)$ or arbitrarily smaller or
larger.  On the other hand, $\dim_{uc}(P')\le\dim_{uc}(P)$ and
$\dim_{lc}(P')\le\dim_{lc}(P)$.

We make the following conjecture, which is easily seen to be
stronger than Conjecture~\ref{c:referee}.

\begin{conjecture}\label{c:ramsey}
For every pair $(d,r)$ of positive integers, there is a poset
$P$ with $\dim(P)=d$ so that if $D$ is the order diagram of
$P$, $E$ is the edge set of $D$ and $\phi:E\rightarrow\{1,2,
\dots,r\}$ is an $r$ coloring of the edges of $D$, then there
is some $\alpha\in\{1,2,\dots,r\}$ so that if we take
$D'$ as the subdiagram of $D$ with edge set
$\{e\in E:\phi(e)=\alpha\}$ and set $P'$ as the suborder of $P$
determined by $D'$, then
$\dim_{uc}(D')=\dim_{lc}(D')=\dim(P)$.
\end{conjecture}

\section{Acknowledgement}

The authors would very much like to thank Penny Haxell for raising the
question answered with Theorem~\ref{t:main-1} in discussions at the 
Graph Theory Conference held in May 2012 at Georgia Tech.
We would also like to express our appreciation to two helpful
referees for their insightful comments and suggestions.  We especially
appreciate the new line of research opened up by the observation
made by one of the referees reflected in Lemma~\ref{l:referee}
and Conjectures~\ref{c:referee} and~\ref{c:ramsey}.


\begin{thebibliography}{9}

\bibitem{b:AgFeTr}
G. Agnarsson, S.~Felsner and W.~T.~Trotter,
The maximum number of edges in a graph of bounded
dimension, with applications to ring theory,
\textit{Discrete Math.} \textbf{201} (1999), 5--19.

\bibitem{b:BarHax}
F. Barrera-Cruz and P. Haxell, A note on Schnyder's theorem,
\textit{Order} \textbf{28} (2011), 221-226.

\bibitem{b:BirHow}
C. Bir\'{o} and D. M. Howard, The first three levels of an
order preserving hamiltonian path in the subset lattice,
\textit{Order} \textbf{28} (2009), 101--107.

\bibitem{b:Boll}
B. Bollob\'{a}s, Colouring lattices, 
\textit{Algebra Universalis} \textbf{7} (1977), 313--314.

\bibitem{b:Desc}
B. Descartes, A three color problem,
\textit{Eureka}, April 1947.  Solution
March 1948.

\bibitem{b:FelTro}
S. Felsner and W. T. Trotter,
Colorings of diagrams of interval orders
and $\alpha$--sequences of sets,
\textit{Discrete Math.} \textbf{144} (1995), 23--31.

\bibitem{b:HosMor}
S. Ho\c{s}ten and W. D. Morris,
The dimension of the complete graph,
\textit{Discrete Math.} \textbf{201}  (1998), 133--139.

\bibitem{b:KriNes}
I. K\v{r}\'{i}\v{z} and J. Ne\v{s}et\v{r}il,
Chromatic number of Hasse diagrams, eyebrows and dimension,
\textit{Order} \textbf{8} (1991), 41--48.

\bibitem{b:NesRod}
J. Ne\v{s}et\v{r}il and V. R\"{o}dl,
A short proof of the existence of highly chromatic graphs without
short cycles,
\textit{J.~Comb. Theory B}~\textbf{27} (1979), 225--227.

\bibitem{b:Schn}
W. Schnyder, Planar graphs and poset dimension,
\textit{Order} \textbf{5} (1989), 323--343.

\bibitem{b:StrTro}
N. Streib and W. T. Trotter,
Hamiltonian cycles and symmetric chains in boolean lattices,
\textit{Graphs and Combinatorics}, to appear.

\end{thebibliography}
\end{document}